\RequirePackage{fix-cm}
\documentclass[smallextended]{svjour3}       
\smartqed  
\usepackage{graphicx}
\usepackage{amsmath}
\usepackage{amsfonts}
\usepackage{color}
\usepackage{lineno}

\begin{document}

\newcommand{\hide}[1]{}
\newcommand{\tbox}[1]{\mbox{\tiny #1}}
\newcommand{\half}{\mbox{\small $\frac{1}{2}$}}
\newcommand{\sinc}{\mbox{sinc}}
\newcommand{\const}{\mbox{const}}
\newcommand{\trc}{\mbox{trace}}
\newcommand{\intt}{\int\!\!\!\!\int }
\newcommand{\ointt}{\int\!\!\!\!\int\!\!\!\!\!\circ\ }
\newcommand{\eexp}{\mbox{e}^}
\newcommand{\bra}{\left\langle}
\newcommand{\ket}{\right\rangle}
\newcommand{\EPS} {\mbox{\LARGE $\epsilon$}}
\newcommand{\ar}{\mathsf r}
\newcommand{\im}{\mbox{Im}}
\newcommand{\re}{\mbox{Re}}
\newcommand{\bmsf}[1]{\bm{\mathsf{#1}}}
\newcommand{\mpg}[2][1.0\hsize]{\begin{minipage}[b]{#1}{#2}\end{minipage}}

\newcommand{\CC}{\mathbb{C}}
\newcommand{\NN}{\mathbb{N}}
\newcommand{\PP}{\mathbb{P}}
\newcommand{\RR}{\mathbb{R}}
\newcommand{\QQ}{\mathbb{Q}}
\newcommand{\ZZ}{\mathbb{Z}}

\newcommand{\p}{\partial}
\renewcommand{\a}{a}
\renewcommand{\b}{\beta}
\renewcommand{\d}{\delta}
\newcommand{\D}{\Delta}
\newcommand{\g}{\gamma}
\newcommand{\G}{\Gamma}
\renewcommand{\th}{\theta}
\renewcommand{\l}{\lambda}
\renewcommand{\L}{\Lambda}
\renewcommand{\O}{\Omega}
\renewcommand{\o}{\omega}
\newcommand{\s}{\sigma}
\newcommand{\e}{\varepsilon}

\title{Analytical and computational study of the variable inverse sum deg index
}


\author{Walter Carballosa         	\and
             J. A. M\'endez-Berm\'udez  	\and
             Jos\'e M. Rodr\'{\i}guez  	\and
             Jos\'e M. Sigarreta
             }


\institute{Walter Carballosa \at
              Department of Mathematics and Statistics, Florida International University, 11200 SW 8th Street
Miami, FL 33199, USA \\
              \email{waltercarb@gmail.com}           
           \and
           J. A. M\'endez-Berm\'udez \at
	     Instituto de F\'{\i}sica, Benem\'erita Universidad Aut\'onoma de Puebla, Apartado Postal J-48, Puebla 72570, Mexico\\
              \email{jmendezb@ifuap.buap.mx}  
              \and
              Jos\'e M. Rodr\'{\i}guez \at
              Departamento de Matem\'aticas, Universidad Carlos III de Madrid, Avenida de la Universidad 30, 28911 Legan\'es, Madrid, Spain \\
              \email{jomaro@math.uc3m.es}  
              \and
              Jos\'e M. Sigarreta \at
              Facultad de Matem\'aticas, Universidad Aut\'onoma de Guerrero, Carlos E. Adame No.54 Col. Garita, 39650 Acapulco Gro., Mexico \\
              \email{josemariasigarretaalmira@hotmail.com}  
              \and
}

\date{Received: date / Accepted: date}

\maketitle

\begin{abstract}
A large number of graph invariants of the form
$\sum_{uv \in E(G)} F(d_u,d_v)$ are studied in mathematical chemistry,
where $uv$ denotes the edge of the graph $G$ connecting the vertices $u$ and $v$,
and $d_u$ is the degree of the vertex $u$.
Among them the variable inverse sum deg index $ISD_a$, with $F(d_u,d_v)=1/(d_u^a+d_v^a)$,
was found to have applicative properties.
The aim of this paper is to obtain new inequalities for the
variable inverse sum deg index, and to characterize graphs extremal with respect to them.
Some of these inequalities generalize and improve previous results for the inverse sum indeg index.
In addition, we computationally validate some of the obtained inequalities on ensembles of random graphs
and show that the ratio $\bra ISD_a(G) \ket/n$ ($n$ being the order of the graph) depends only on the average 
degree $\bra d \ket$.
\end{abstract}

\section{Introduction}

Topological indices are parameters associated with chemical compounds that associate the chemical structure with several physical, chemical or biological properties.

A family of degree--based topological indices, named \emph{Adriatic indices}, was put
forward in \cite{VG,V2}. Twenty of them were selected as significant predictors.
One of them, the \emph{inverse sum indeg} index, $ISI$, was singled out
in \cite{VG,V2} as a significant predictor of total surface area of octane isomers.
This index is defined as
$$
ISI(G) = \sum_{uv\in E(G)} \frac{d_u\,d_v}{d_u + d_v}
= \sum_{uv\in E(G)} \frac{1}{\frac{1}{d_u} + \frac{1}{d_v}}\,,
$$
where $uv$ denotes the edge of the graph $G$ connecting the vertices $u$ and $v$,
and $d_u$ is the degree of the vertex $u$.
In the last years there is an increasing interest in the mathematical
properties of this index (see, e.g., \cite{ChenDeng,FAD,MMM,GRS,Mingqiang,RRS,SSV}).

We study here the properties of the \emph{variable inverse sum deg index} defined, for each $a \in \RR$, as
$$
ISD_a(G)
= \sum_{uv \in E(G)} \frac{1}{d_u^a + d_v^a} \,.
$$
Note that $ISD_{-1}$ is the inverse sum indeg index $ISI$.

The variable inverse sum deg index $ISD_{-1.950}$ was selected in \cite{V4} as a significant predictor of standard enthalpy of formation.

The idea behind the variable molecular descriptors is that the variables are determined during the
regression so that the standard error of estimate for a particular studied property is as
small as possible (see, e.g., \cite{MN}).

The aim of this paper is to obtain new inequalities for the
variable inverse sum deg index, and to characterize graphs extremal with respect to them.
Some of these inequalities generalize and improve previous results for the inverse sum indeg index.
Also, we want to remark that many previous results are proved
for connected graphs, but our inequalities hold for both connected and non-connected
graphs.

Throughout this paper, $G=(V (G),E (G))$ denotes an undirected finite simple (without multiple edges and loops)
graph without isolated vertices.
We denote by $n$, $m$, $\D$ and $\d$ the cardinality of the set of vertices of $G$,
the cardinality of the set of edges of $G$,
its maximum degree and its minimum degree, respectively.
Thus, we have $1 \le \d \le \D < n$.
We denote by $N(u)$ the set of neighbors of the vertex $u \in V(G)$.

\section{Inequalities for the $ISD_a$ index}

\begin{proposition} \label{p:m}
If $G$ is a graph with minimum degree $\d$, maximum degree $\D$ and $m$ edges, and $a \in \RR$, then
$$
\begin{aligned}
\frac{m}{2\D^a}
\le ISD_a(G)
\le \frac{m}{2\d^a} \,,
\qquad & \text{if } \, a > 0,
\\
\frac{m}{2\d^a}
\le ISD_a(G)
\le \frac{m}{2\D^a} \,,
\qquad & \text{if } \, a < 0.
\end{aligned}
$$
The equality in each bound is attained if and only if $G$ is regular.
\end{proposition}

\begin{proof}
If $a>0$, then $2 \d^{a} \le d_u^{a}+d_v^{a} \le 2 \D^{a}$ and
$$
\begin{aligned}
ISD_a(G)
& = \sum_{uv \in E(G)} \frac{1}{d_u^{a}+d_v^{a}}
\le \sum_{uv \in E(G)} \frac{1}{2\d^a}
= \frac{m}{2\d^a} \,,
\\
ISD_a(G)
& = \sum_{uv \in E(G)} \frac{1}{d_u^{a}+d_v^{a}}
\ge \sum_{uv \in E(G)} \frac{1}{2\D^a}
= \frac{m}{2\D^a} \,.
\end{aligned}
$$

If $a<0$, then the previous argument gives the converse inequalities.

\smallskip

If $G$ is a regular graph, then the lower and upper bounds are the same, and they are equal to $ISD_a(G)$.

Assume now that the equality in some bound is attained.
Thus, by the previous argument we have either $d_u=d_v=\d$ for every $uv\in E(G)$, or $d_u=d_v=\D$ for every $uv\in E(G)$.
Hence, $G$ is regular.
\end{proof}

In 1998 Bollob\'{a}s and Erd\"{o}s \cite{BE} generalized the Randi\'{c} index by replacing $1/2$ by any real number.
Thus, for $a \in \mathbb{R}\setminus \{0\}$, the \emph{general Randi\'{c} index} of a graph $G$ is defined as
$$
R_\a(G) = \sum_{uv\in E(G)} (d_u d_v)^\a .
$$
The general Randi\'{c} index, also called \emph{variable Zagreb index} in 2004 by Mili{c}evi\'{c} and Nikoli\'{c} \cite{MN}, has been extensively studied \cite{LG}.
Note that $R_{-1/2}$ is the usual Randi\'c index, $R_{1}$ is the second Zagreb index $M_2$,
$R_{-1}$ is the modified Zagreb index \cite{NKMT}, etc.
In Randi\'{c}'s original paper \cite{R}, in addition to the particular case $a=-1/2$, also the index with $a=-1$ was briefly considered.

\medskip

The next result relates the $ISD_a$ and $R_{-a}$ indices.

\begin{theorem} \label{t:m2}
If $G$ is a graph with minimum degree $\d$ and maximum degree $\D$, and $a \in \RR$, then
$$
\begin{aligned}
\frac12 \, \d^{a} R_{-a}(G)
\le ISD_a(G)
\le \frac12 \, \D^{a} R_{-a}(G),
\qquad & \text{if } \, a > 0,
\\
\frac12 \, \D^{a} R_{-a}(G)
\le ISD_a(G)
\le \frac12 \, \d^{a} R_{-a}(G),
\qquad & \text{if } \, a < 0.
\end{aligned}
$$
The equality in each bound is attained if and only if $G$ is regular.
\end{theorem}

\begin{proof}
We have
$$
ISD_a(G)
= \sum_{uv \in E(G)} \frac{1}{d_u^{a}+d_v^{a}}
= \sum_{uv \in E(G)} \frac{(d_u d_v)^{-a}}{d_u^{-a}+d_v^{-a}} \,.
$$

If $a>0$, then $2 \D^{-a} \le d_u^{-a}+d_v^{-a} \le 2 \d^{-a}$, and
$$
\begin{aligned}
\frac12 \, \d^{a} R_{-a}(G)
= \sum_{uv \in E(G)} \frac{(d_u d_v)^{-a}}{2\d^{-a}}
\le \sum_{uv \in E(G)} \frac{(d_u d_v)^{-a}}{d_u^{-a}+d_v^{-a}}
\\
\le \sum_{uv \in E(G)} \frac{(d_u d_v)^{-a}}{2\D^{-a}}
= \frac12 \, \D^{a} R_{-a}(G) .
\end{aligned}
$$

If $a<0$, then the previous argument gives the converse inequalities.

\smallskip

If $G$ is a regular graph, then the lower and upper bounds are the same, and they are equal to $ISD_a(G)$.

Assume now that the equality in some bound is attained.
Thus, by the previous argument we have either $d_u=d_v=\d$ for every $uv\in E(G)$, or $d_u=d_v=\D$ for every $uv\in E(G)$.
Hence, $G$ is regular.
\end{proof}

The following result relates the $ISD_a$ and $ISD_{-a}$ indices.

\begin{theorem} \label{t:a-a}
If $G$ is a graph with minimum degree $\d$ and maximum degree $\D$, and $a \in \RR$, then
$$
\begin{aligned}
\D^{-2a} ISD_{-a}(G)
\le ISD_a(G)
\le \d^{-2a} ISD_{-a}(G),
\qquad & \text{if } \, a > 0,
\\
\d^{-2a} ISD_{-a}(G)
\le ISD_a(G)
\le \D^{-2a} ISD_{-a}(G),
\qquad & \text{if } \, a < 0.
\end{aligned}
$$
The equality in each bound is attained if and only if $G$ is regular.
\end{theorem}

\begin{proof}
We have
$$
ISD_a(G)
= \sum_{uv \in E(G)} \frac{1}{d_u^{a}+d_v^{a}}
= \sum_{uv \in E(G)} \frac{(d_u d_v)^{-a}}{d_u^{-a}+d_v^{-a}} \,.
$$

Similarly, we obtain the result if $a>0$.
$$
\begin{aligned}
\D^{-2a} ISD_{-a}(G)
= \sum_{uv \in E(G)} \frac{\D^{-2a}}{d_u^{-a}+d_v^{-a}}
\le ISD_a(G)
\\
\le \sum_{uv \in E(G)} \frac{\d^{-2a}}{d_u^{-a}+d_v^{-a}}
= \d^{-2a} ISD_{-a}(G) .
\end{aligned}
$$

If $a<0$, then the previous argument gives the converse inequalities.

\smallskip

If $G$ is a regular graph, then the lower and upper bounds are the same, and they are equal to $ISD_a(G)$.

If the equality in some bound is attained, by the previous argument we have either $d_u=d_v=\d$ for every $uv\in E(G)$, or $d_u=d_v=\D$ for every $uv\in E(G)$.
Therefore, $G$ is regular.
\end{proof}

The \emph{general sum-connectivity index} was defined in \cite{ZT2} as
$$
\chi_{_{\a}}(G) = \sum_{uv\in E(G)} (d_u+ d_v)^\a\,.
$$
Note that $\chi_{_{1}}$ is the first Zagreb index $M_1$, $2\chi_{_{-1}}$ is the harmonic index $H$,
$\chi_{_{-1/2}}$ is the sum-connectivity index $\chi$, etc.

\smallskip

The following result relates the variable inverse sum deg
and the general sum-connectivity indices.

\begin{theorem} \label{t:chi}
If $G$ is a graph and $a \in \RR \setminus \{0,1\}$, then
\begin{eqnarray}
\label{Eq1}
\chi_{_{-a}}(G) < ISD_a(G) \le 2^{a-1} \chi_{_{-a}}(G),
\qquad & \text{if } \, a > 1,
\\
\label{Eq2}
2^{a-1} \chi_{_{-a}}(G) \le ISD_a(G) < \chi_{_{-a}}(G),
\qquad & \text{if } \, 0< a < 1,
\\
\label{Eq3}
ISD_a(G) \le 2^{a-1} \chi_{_{-a}}(G),
\qquad & \text{if } \, a < 0.
\end{eqnarray}
The equality in the first or third upper bound or in the second lower bound is attained if and only if each connected component of $G$ is regular.
\end{theorem}

\begin{proof}
We want to compute the minimum and maximum values of the function
$f: \RR^+ \times \RR^+ \to \RR^+$
given by
$$
f(x,y)
= \frac{(x+y)^{a}}{x^{a}+y^{a}} \,.
$$
In order to do that, we are going to compute the extremal values of $g(x,y) = (x+y)^{a}$ with the restrictions
$h(x,y)= x^{a} +y^{a}=1$, $x,y>0$.
If $(x,y)$ is a critical point, then there exists $\l \in \RR$ such that
$$
\begin{aligned}
a(x+y)^{a-1}
& = \l\, a\, x^{a-1},
\\
a(x+y)^{a-1}
& = \l\, a\, y^{a-1},
\end{aligned}
$$
and so, $x=y$;
this fact and the equality $x^{a} +y^{a}=1$ give $x=y= 2^{-1/a}$ and $g(2^{-1/a},2^{-1/a}) = 2^{a-1}$.

If $a>0$ and $x \to 0^+$ (respectively, $y \to 0^+$), then $y \to 1$ (respectively, $x \to 1$) and $g(x,y) \to 1$.

If $a>1$, then $1 < g(x,y) \le 2^{a-1}$ and the upper bound is attained if and only if $x=y$.
By homogeneity, we have $1 < f(x,y) \le 2^{a-1}$ for every $x,y>0$ and the upper bound is attained if and only if $x=y$.

If $0<a<1$, then $2^{a-1} \le g(x,y) < 1$ and the lower bound is attained if and only if $x=y$.
Thus, $2^{a-1} \le f(x,y) < 1$ for every $x,y>0$ and the lower bound is attained if and only if $x=y$.

If $a<0$, then $x,y>1$.
If $x \to 1^+$ (respectively, $y \to 1^+$), then $y \to \infty$ (respectively, $x \to \infty$) and $g(x,y) \to 0$.
Hence, $0 < g(x,y) \le 2^{a-1}$ and the upper bound is attained if and only if $x=y$.
Thus, $0 < f(x,y) \le 2^{a-1}$ for every $x,y>0$ and the upper bound is attained if and only if $x=y$.

Note that if $c_a \le f(x,y) \le C_a$, then
$$
c_a \frac{1}{( d_u + d_v )^{a}}
\le \frac{1}{d_u^{a} + d_v^{a}}
\le C_a \frac{1}{( d_u + d_v )^{a}}
$$
for every $uv\in E(G)$ and, consequently, $c_\a \chi_{_{-a}}(G) \le ISD_a(G) \le C_\a \chi_{_{-a}}(G)$.
These facts give the inequalities.

\smallskip

If $G$ is a connected $\d$-regular graph with $m$ edges, then
$$
2^{a-1} \chi_{_{-a}}(G)
= 2^{a-1} (2\d)^{-a} m
= \frac{m}{2\d^{a}}
= ISD_a(G).
$$
By linearity, the equality
$2^{a-1} \chi_{_{-a}}(G)
= ISD_a(G)$
also holds if each connected component of $G$ is regular.

Assume now that the equality in the first or third upper bound or in the second lower bound is attained.
Thus, the previous argument gives that $d_u=d_v$
for every $uv\in E(G)$ and, consequently,
each connected component of $G$ is regular.
\end{proof}

Note that Theorem \ref{t:chi}, with $a =-1$, gives $ISI(G) \le M_{1}(G)/4$, a known inequality (see \cite[Theorem 4]{SSV}).
Hence, Theorem \ref{t:chi} generalizes \cite[Theorem 4]{SSV}.

\begin{remark}
Note that if we take limits as $a \to 1$ in Theorem \ref{t:chi}, then we obtain by continuity the trivial equality
$ISD_1(G) = \chi_{_{-1}}(G)$.
\end{remark}

The geometric-arithmetic index was introduced in \cite{VF} as
$$
GA(G) = \sum_{uv\in E(G)}\frac{2\sqrt{d_u d_v}}{d_u + d_v} \,.
$$
Although it was introduced in $2009$, there are many papers dealing with this index
(see, e.g., \cite{DGF}, \cite{DGF2}, \cite{MR}, \cite{MH}, \cite{PST}, \cite{RRS2}, \cite{RS2}, \cite{RS3}, \cite{S}, \cite{VF} and the references therein).
The predicting ability of the $GA$ index compared with
Randi\'c index is reasonably better (see \cite[Table 1]{DGF}).
The graphic in \cite[Fig.7]{DGF} (from \cite[Table 2]{DGF}, \cite{TRC}) shows
that there exists a good linear correlation between $GA$ and the heat of formation of benzenoid hydrocarbons
(the correlation coefficient is equal to $0.972$).
Furthermore, the improvement in
prediction with $GA$ index comparing to Randi\'c index in the case of standard
enthalpy of vaporization is more than 9$\%$. That is why one can think that $GA_1$ index
should be considered in the QSPR/QSAR researches.

\smallskip

The following result relates the variable inverse sum deg
and the geometric-arithmetic indices.

\begin{theorem} \label{t:ga}
If $G$ is a graph and $a \in \RR$, then
$$
\begin{aligned}
ISD_a(G) \ge \frac12 \, \D^{-a} GA(G),
\qquad & \text{if } \, a > 0,
\\
ISD_a(G) \ge \frac12 \, \d^{-a} GA(G),
\qquad & \text{if } \, a < 0.
\end{aligned}
$$
The equality in each bound is attained if and only if $G$ is a regular graph.
\end{theorem}

\begin{proof}
We are going to compute the minimum and maximum values of the function
$V: [\d,\D] \times [\d,\D] \to \RR^+$
given by
$$
V(x,y)
= \frac{x+y}{2\sqrt{xy}\,(x^{a}+y^{a})} \,.
$$
We have
$$
\begin{aligned}
\frac{\p V}{\p x}\,(x,y)
& = \frac{1}{2\sqrt{y}} \; \frac{x^{1/2}(x^{a}+y^{a})-(x+y)\big( \frac12\,x^{-1/2}(x^{a}+y^{a}) + x^{1/2}ax^{a-1}\big)}{x(x^{a}+y^{a})^2}
\\
& = \frac{2x(x^{a}+y^{a})-(x+y)\big( x^{a}+y^{a} + 2ax^{a}\big)}{4x^{3/2}y^{1/2}(x^{a}+y^{a})^2}
\\
& = \frac{(x-y)(x^{a}+y^{a})-2a(x+y)x^{a}}{4x^{3/2}y^{1/2}(x^{a}+y^{a})^2}
\,.
\end{aligned}
$$

Assume first that $a>0$.
By symmetry, we can assume that $x \le y$.
Thus,
$\p V/\p x (x,y)<0$ for $\d \le x \le y \le \D$, and so,
$$
V(x,y)
\ge V(y,y)
= \frac{1}{2y^{a}}
\ge \frac12 \, \D^{-a},
$$
and the equality in the bound is attained if and only if $x=y=\D$.
Hence,
$$
\begin{aligned}
\frac1{d_u^{a}+d_v^{a}}
& \ge \frac12 \, \D^{-a} \, \frac{2\sqrt{d_ud_v}}{d_u+d_v} \,,
\\
ISD_a(G)
= \sum_{uv \in E(G)} \frac1{d_u^{a}+d_v^{a}}
& \ge \frac12 \, \D^{-a} \!\!\!\!\! \sum_{uv \in E(G)} \!\! \frac{2\sqrt{d_ud_v}}{d_u+d_v}
= \frac12 \, \D^{-a} GA(G) ,
\end{aligned}
$$
and the equality in the bound is attained if and only if $d_u=d_v=\D$ for every $uv \in E(G)$, i.e., $G$ is a regular graph.

\medskip

Assume now that $a<0$.
We can assume that $y \le x$.
Thus,
$\p V/\p x (x,y)>0$ for $\d \le y \le x \le \D$, and so,
$$
V(x,y)
\ge V(y,y)
= \frac{1}{2y^{a}}
\ge \frac12 \, \d^{-a},
$$
and the equality in the bound is attained if and only if $x=y=\D$.
Hence,
$$
\begin{aligned}
\frac1{d_u^{a}+d_v^{a}}
& \ge \frac12 \, \d^{-a} \, \frac{2\sqrt{d_ud_v}}{d_u+d_v} \,,
\\
ISD_a(G)
= \sum_{uv \in E(G)} \frac1{d_u^{a}+d_v^{a}}
& \ge \frac12 \, \d^{-a} \!\!\!\!\! \sum_{uv \in E(G)} \!\! \frac{2\sqrt{d_ud_v}}{d_u+d_v}
= \frac12 \, \d^{-a} AG(G) ,
\end{aligned}
$$
and the equality in the bound is attained if and only if $d_u=d_v=\d$ for every $uv \in E(G)$, i.e., $G$ is a regular graph.
\end{proof}

As an inverse variant of the geometric-arithmetic index, in 2015, the arithmetic-geometric index was introduced in \cite{SK1} as
$$
AG(G) = \sum_{uv\in E(G)}\frac{d_u + d_v}{2\sqrt{d_u d_v}} \,.
$$
In \cite{MRSS} it is shown that the arithmetic-geometric index has a good predictive power
for entropy of octane isomers.
The paper \cite{ZTC} studied spectrum and energy of arithmetic-geometric
matrix, in which the sum of all elements is equal to 2$AG$. Other bounds
of the arithmetic-geometric energy of graphs appeared in \cite{GG}, \cite{DG}.
The paper \cite{VP} studies optimal $AG$-graphs for several classes of graphs.
In \cite{CGMPP}, \cite{CWTW}, \cite{MRSS} and \cite{RSST} there are more bounds on the $AG$ index.

\smallskip

The following result relates the variable inverse sum deg
and the arithmetic-geometric indices.

\begin{theorem} \label{t:ag}
If $G$ is a graph and $a \in \RR$, then
$$
\begin{aligned}
ISD_a(G) \le \frac12 \, \d^{-a} AG(G),
\qquad & \text{if } \, a > 0,
\\
ISD_a(G) \le \frac12 \, \D^{-a} AG(G),
\qquad & \text{if } \, a < 0.
\end{aligned}
$$
The equality in each bound is attained if and only if $G$ is a regular graph.
\end{theorem}

\begin{proof}
We are going to compute the minimum and maximum values of the function
$U: [\d,\D] \times [\d,\D] \to \RR^+$
given by
$$
U(x,y)
= \frac{2\sqrt{xy}}{(x+y)(x^{a}+y^{a})} \,.
$$
We have
$$
\begin{aligned}
\frac{\p U}{\p x}\,(x,y)
& = \sqrt{y} \; \frac{x^{-1/2}(x+y)(x^{a}+y^{a})-2x^{1/2}\big( x^{a}+y^{a} + (x+y)ax^{a-1}\big)}{(x+y)^2(x^{a}+y^{a})^2}
\\
& = \sqrt{y} \; \frac{(x+y)(x^{a}+y^{a})-2\big( x(x^{a}+y^{a}) + (x+y)ax^{a} \big)}{\sqrt{x} \,(x+y)^2(x^{a}+y^{a})^2}
\\
& = \sqrt{y} \; \frac{(y-x)(x^{a}+y^{a})-2(x+y)ax^{a}}{\sqrt{x} \,(x+y)^2(x^{a}+y^{a})^2}
\,.
\end{aligned}
$$

Assume first that $a>0$.
By symmetry, we can assume that $x \ge y$.
Thus,
$\p U/\p x (x,y)<0$ for $\d \le y \le x \le \D$, and so,
$$
U(x,y)
\le U(y,y)
= \frac{1}{2y^{a}}
\le \frac12 \, \d^{-a},
$$
and the equality in the bound is attained if and only if $x=y=\d$.
Hence,
$$
\begin{aligned}
\frac1{d_u^{a}+d_v^{a}}
& \le \frac12 \, \d^{-a} \, \frac{d_u+d_v}{2\sqrt{d_ud_v}} \,,
\\
ISD_a(G)
= \sum_{uv \in E(G)} \frac1{d_u^{a}+d_v^{a}}
& \le \frac12 \, \d^{-a} \!\!\!\!\! \sum_{uv \in E(G)} \!\! \frac{d_u+d_v}{2\sqrt{d_ud_v}}
= \frac12 \, \d^{-a} AG(G) ,
\end{aligned}
$$
and the equality in the bound is attained if and only if $d_u=d_v=\d$ for every $uv \in E(G)$, i.e., $G$ is a regular graph.

\medskip

Assume now that $a<0$.
We can assume that $x \le y$.
Thus,
$\p U/\p x (x,y)>0$ for $\d \le x \le y \le \D$, and so,
$$
U(x,y)
\le U(y,y)
= \frac{1}{2y^{a}}
\le \frac12 \, \D^{-a},
$$
and the equality in the bound is attained if and only if $x=y=\D$.
Hence,
$$
\begin{aligned}
\frac1{d_u^{a}+d_v^{a}}
& \le \frac12 \, \D^{-a} \, \frac{d_u+d_v}{2\sqrt{d_ud_v}} \,,
\\
ISD_a(G)
= \sum_{uv \in E(G)} \frac1{d_u^{a}+d_v^{a}}
& \le \frac12 \, \D^{-a} \!\!\!\!\! \sum_{uv \in E(G)} \!\! \frac{d_u+d_v}{2\sqrt{d_ud_v}}
= \frac12 \, \D^{-a} AG(G) ,
\end{aligned}
$$
and the equality in the bound is attained if and only if $d_u=d_v=\D$ for every $uv \in E(G)$, i.e., $G$ is a regular graph.
\end{proof}

Mili\v{c}evi\'c and Nikoli\'c defined in \cite{MN} the \emph{variable first Zagreb index} as
$$
M_1^{\a}(G) = \sum_{u\in V(G)} d_u^{\a},
$$
with $\a \in \RR$.
Note that $M_{1}^2$ is the first Zagreb index $M_1$,
$M_{1}^{-1}$ is the the inverse index $ID$,
$M_{1}^{-1/2}$ is the zeroth-order Randi\'c index,
$M_{1}^3$ is the forgotten index $F$, etc.

\begin{theorem} \label{t:m1}
If $G$ is a graph with $m$ edges, and $a \in \RR$, then
\begin{eqnarray}
\label{Eq4}
ISD_a(G) + M_1^{a+1}(G)
\ge \frac52 \,m ,
\qquad & \text{if } \, a > 0,
\\
\label{Eq5}
ISD_a(G) + M_1^{a+1}(G)
\ge 2m ,
\qquad & \text{if } \, a < 0.
\end{eqnarray}
The equality in the first bound is attained if and only if $G$ is a union of path graphs $P_2$.
\end{theorem}

\begin{proof}
Recall that we have for any function $h$
$$
\sum_{uv \in E(G)} \big( h(d_u)+ h(d_v) \big)
= \sum_{u \in V(G)} d_u h(d_u) .
$$
In particular,
$$
\sum_{uv \in E(G)} \big( d_u^{a}+d_v^{a} \big)
= \sum_{u \in V(G)} d_u^{a+1}
= M_1^{a+1}(G).
$$

The function $f(x)=x+1/x$ is strictly decreasing on $(0,1]$ and strictly increasing on $[1,\infty)$, and so, $f(x) \ge f(1) = 2$ for every $x>0$.
Hence,
$$
\begin{aligned}
\frac{1}{d_u^{a}+d_v^{a}} + d_u^{a}+d_v^{a}
& \ge 2,
\\
ISD_a(G) + M_1^{a+1}(G)
& \ge 2m .
\end{aligned}
$$

If $a>0$, then $d_u^{a}+d_v^{a} \ge 2$ and
$$
\begin{aligned}
\frac{1}{d_u^{a}+d_v^{a}} + d_u^{a}+d_v^{a}
& \ge f(2)
= \frac52 \,,
\\
ISD_a(G) + M_1^{a+1}(G)
& \ge \frac52 \,m .
\end{aligned}
$$

The previous argument gives that the equality in this bound is attained if and only if $d_u=d_v=1$ for every $uv \in E(G)$,
i.e., $G$ is a union of path graphs $P_2$.
\end{proof}

\begin{theorem} \label{t:m1d}
Let $G$ be a graph with minimum degree $\d$ and $m$ edges, and $a \in \RR$.

$(1)$ If $a > 0$, then
$$
\begin{aligned}
ISD_a(G) + M_1^{a+1}(G)
\ge \Big( 2\d^a + \frac1{2\d^a} \Big) m .
\end{aligned}
$$

$(2)$ If $\d>1$ and $a \le -\log 2/\log \d$, then
$$
\begin{aligned}
ISD_a(G) + M_1^{a+1}(G)
\ge \Big( 2\d^a + \frac1{2\d^a} \Big) m .
\end{aligned}
$$

The equality in each bound is attained if and only if $G$ is regular.
\end{theorem}

\begin{proof}
If $a>0$, then
$d_u^{a}+d_v^{a} \ge 2 \d^a \ge 2 >1$.
The argument in the proof of Theorem \ref{t:m1} gives
$$
\begin{aligned}
\frac{1}{d_u^{a}+d_v^{a}} + d_u^{a}+d_v^{a}
& \ge f(2 \d^a)
= 2\d^a + \frac1{2\d^a} \,,
\\
ISD_a(G) + M_1^{a+1}(G)
& \ge \Big( 2\d^a + \frac1{2\d^a} \Big) m .
\end{aligned}
$$

If $\d>1$ and $a \le -\log 2 / \log \d < 0$, then
$2 \d^a \le 1$
and
$d_u^{a}+d_v^{a} \le 2 \d^a \le 1$.
Thus,
$$
\begin{aligned}
\frac{1}{d_u^{a}+d_v^{a}} + d_u^{a}+d_v^{a}
& \ge f(2 \d^a)
= 2\d^a + \frac1{2\d^a} \,,
\\
ISD_a(G) + M_1^{a+1}(G)
& \ge \Big( 2\d^a + \frac1{2\d^a} \Big) m .
\end{aligned}
$$

The previous argument gives that the equality in each bound is attained if and only if $d_u^{a}+d_v^{a} = 2 \d^a$ for every $uv \in E(G)$,
i.e., $d_u=d_v=\d$ for every $uv \in E(G)$;
and this holds if and only if $G$ is regular.
\end{proof}

\begin{theorem} \label{t:m1D}
Let $G$ be a graph with maximum degree $\D$ and $m$ edges, and $a > 0$.
Then
$$
\begin{aligned}
ISD_a(G) + M_1^{a+1}(G)
\le \Big( 2\D^a + \frac1{2\D^a} \Big) m ,
\end{aligned}
$$
and the equality in the bound is attained if and only if $G$ is regular.
\end{theorem}

\begin{proof}
If $a>0$, then
$1 < 2 \le d_u^{a}+d_v^{a} \le 2 \D^a$.
The argument in the proof of Theorem \ref{t:m1} gives
$$
\begin{aligned}
\frac{1}{d_u^{a}+d_v^{a}} + d_u^{a}+d_v^{a}
& \le f(2 \D^a)
= 2\D^a + \frac1{2\D^a} \,,
\\
ISD_a(G) + M_1^{a+1}(G)
& \le \Big( 2\D^a + \frac1{2\D^a} \Big) m .
\end{aligned}
$$

The previous argument gives that the equality in the bound is attained if and only if $d_u^{a}+d_v^{a} = 2 \D^a$ for every $uv \in E(G)$,
i.e., $d_u=d_v=\D$ for every $uv \in E(G)$;
and this holds if and only if $G$ is regular.
\end{proof}

We need the following well known result, that provides a converse of the
Cauchy-Schwarz inequality (see, e.g., \cite[Lemma 3.4]{MRS}).

\begin{lemma} \label{l:PS2}
If $a_j,b_j\ge 0$ and $\omega b_j \le a_j \le \O b_j$ for $1\le j \le k$, then
$$
\Big(\sum_{j=1}^k a_j^2 \Big)^{1/2} \Big(\sum_{j=1}^k b_j^2 \Big)^{1/2}
\leq \frac12 \Big(\,\sqrt{\frac{\O}{\omega}}+ \sqrt{\frac{\omega}{\O}} \;\,\Big)\sum_{j=1}^k a_j\,b_j\,.
$$
If $a_j>0$ for some $1\le j \le k$, then the equality holds if and
only if $\omega=\O$ and $a_j=\omega b_j$ for every $1\le j \le k$.
\end{lemma}

Recall that a $(\D,\d)$-biregular graph is a bipartite graph for which any vertex in one side of the given bipartition has degree $\D$
and any vertex in the other side of the bipartition has degree $\d$.

\begin{theorem} \label{t:M1bis}
If $G$ is a graph with $m$ edges, maximum degree $\D$ and minimum degree $\d$,
and $a \in \RR \setminus \{0\}$, then
\begin{equation}
\label{Eq6}
m^2
\le ISD_a(G) M_1^{a+1}(G)
\le \frac{(\D^{a}+\d^{a})^2}{4\D^{a}\d^{a}} \, m^2 .
\end{equation}
The equality in the upper bound is attained if and only if $G$ is regular.
The equality in the lower bound is attained if $G$ is regular or biregular.
Furthermore, if $G$ is a connected graph, then the equality in the lower bound is attained if and only if $G$ is a regular or biregular graph.
\end{theorem}

\begin{proof}
Cauchy-Schwarz inequality gives
$$
\begin{aligned}
m^2
= \Big( \sum_{uv\in E(G)} \frac1{\sqrt{d_u^a + d_v^a}} \, \sqrt{d_u^a + d_v^a} \,\; \Big)^2
& \le \sum_{uv\in E(G)} \frac1{d_u^a + d_v^a} \sum_{uv\in E(G)} \big( d_u^a + d_v^a \big)
\\
& = ISD_a(G) M_1^{a+1}(G) .
\end{aligned}
$$
If $a>0$, then
$$
\begin{aligned}
2\d^a
\le d_u^a + d_v^a
& = \frac{ \sqrt{d_u^a + d_v^a}}{\frac1{\sqrt{d_u^a + d_v^a}}}
\le 2\D^a.
\end{aligned}
$$
If $a<0$, then
$$
\begin{aligned}
2\D^a
\le d_u^a + d_v^a
& = \frac{ \sqrt{d_u^a + d_v^a}}{\frac1{\sqrt{d_u^a + d_v^a}}}
\le 2\d^a.
\end{aligned}
$$
Lemma \ref{l:PS2} gives for every $a \neq 0$
$$
\begin{aligned}
m^2
& = \Big( \sum_{uv\in E(G)} \frac1{\sqrt{d_u^a + d_v^a}} \, \sqrt{d_u^a + d_v^a} \,\; \Big)^2 \\
& \ge \frac{\sum_{uv\in E(G)} \frac1{d_u^a + d_v^a} \sum_{uv\in E(G)} \big( d_u^a + d_v^a \big)}{\frac14\big( \frac{\D^{a/2}}{\d^{a/2}} + \frac{\d^{a/2}}{\D^{a/2}} \big)^2}
\\
& = \frac{4\D^{a}\d^{a}}{(\D^{a}+\d^{a})^2}\, ISD_a(G) M_1^{a+1}(G) .
\end{aligned}
$$
\indent
If $G$ is a regular graph, then the lower and upper bounds are the same, and they are equal to $ISD_a(G)M_1^{a+1}(G)$.

Assume now that the equality in the upper bound is attained.
Lemma \ref{l:PS2} gives $2\D^a=2\d^a$ and so,
$\D=\d$ and $G$ is regular.

If $G$ is a regular or biregular graph, then
\begin{equation}
\label{Eq7}
ISD_a(G) M_1^{a+1}(G)
= \frac{m}{\D^{a}+\d^{a}}\,(\D^{a}+\d^{a}) \, m
= m^2,
\end{equation}
and the lower bound is attained.

Assume now that $G$ is a connected graph.
By Cauchy-Schwarz inequality, the equality in the lower bound
is attained if only if there exists a constant $\eta$ such that,
for every $uv\in E(G)$,
\begin{equation} \label{eq:450}
\frac1{\sqrt{d_u^a + d_v^a}} = \eta \, \sqrt{d_u^a + d_v^a}\,,
\qquad
d_u^a + d_v^a=\eta^{-1} .
\end{equation}
If $uv,uw\in E(G)$, then
$$
\eta^{-1} = d_u^a + d_v^a
=d_u^a + d_w^a \,,
$$
and $d_w =d_v$, since $h(t)=t^a$ is a one to one function.
Thus, we conclude that \eqref{eq:450} is equivalent to the following:
for each vertex $u\in V(G)$, every neighbor of $u$ has the same degree.
Since $G$ is connected, this holds if and only if $G$ is regular or biregular.
\end{proof}

\section{Computational study of the $ISD_a$ index on random graphs}
\label{statistics}

Here we follow a recently introduced approach under which topological indices are applied to ensembles
of random graphs. Thus instead of computing the index of a single graph, the index average
value over a large number of random graphs is measured as a function of the random graph parameters;
see the application of this approach to Erd\H{o}s-R\'{e}nyi graphs and random regular graphs 
in~\cite{MMRS20,AMRS20,AIMS20,MMRS21}.

We consider random graphs $G$ from the standard Erd\H{o}s-R\'{e}nyi (ER) model $G(n,p)$, i.e., $G$
has $n$ vertices and each edge appears independently with probability $p \in (0,1)$.
The computational study of the $ISD_a$ index we perform below is justified by the random nature of the 
ER model: since a given parameter pair $(n,p)$ represents an infinite--size ensemble of ER graphs, the 
computation of a $ISD_a$ index on a single ER graph is irrelevant. In contrast, the computation
of $\left< ISD_a \right>$ over a large ensemble of ER graphs, all characterized by the same parameter pair 
$(n,p)$, may provide useful average information about the ensemble.
Also, we extend some of the inequalities derived in the previous Section to index average values.

\subsection{Scaling of the average $ISD_a$ index on random graphs}

In Fig.~\ref{Fig01}(a) we plot the average variable inverse sum deg index $\left< ISD_a(G) \right>$ as a 
function of the probability $p$ of ER graphs of size $n=1000$.
There, we show curves for $a\in[-2,2]$. As a reference we plot in different colors the curves corresponding
to $a=-1$ (blue), $a=0$ (red), and $a=1$ (green). Recall that $ISD_{-1}(G)=ISI(G)$ and 
$ISD_1(G)=\chi_{_{-1}}(G)$. While for $a=0$, $\left< ISD_0(G) \right>$ gives half of the average 
number of edges of the ER graph; that is,
\begin{equation}
\label{E}
\left< ISD_0(G) \right> = \sum_{uv \in E(G)} \frac{1}{d_u^0 + d_v^0} = \frac{1}{2} |E(G)| = \frac{1}{4} n(n-1)p \ .
\end{equation}

\begin{figure}[t]
\begin{center} 
\includegraphics[width=0.9\textwidth]{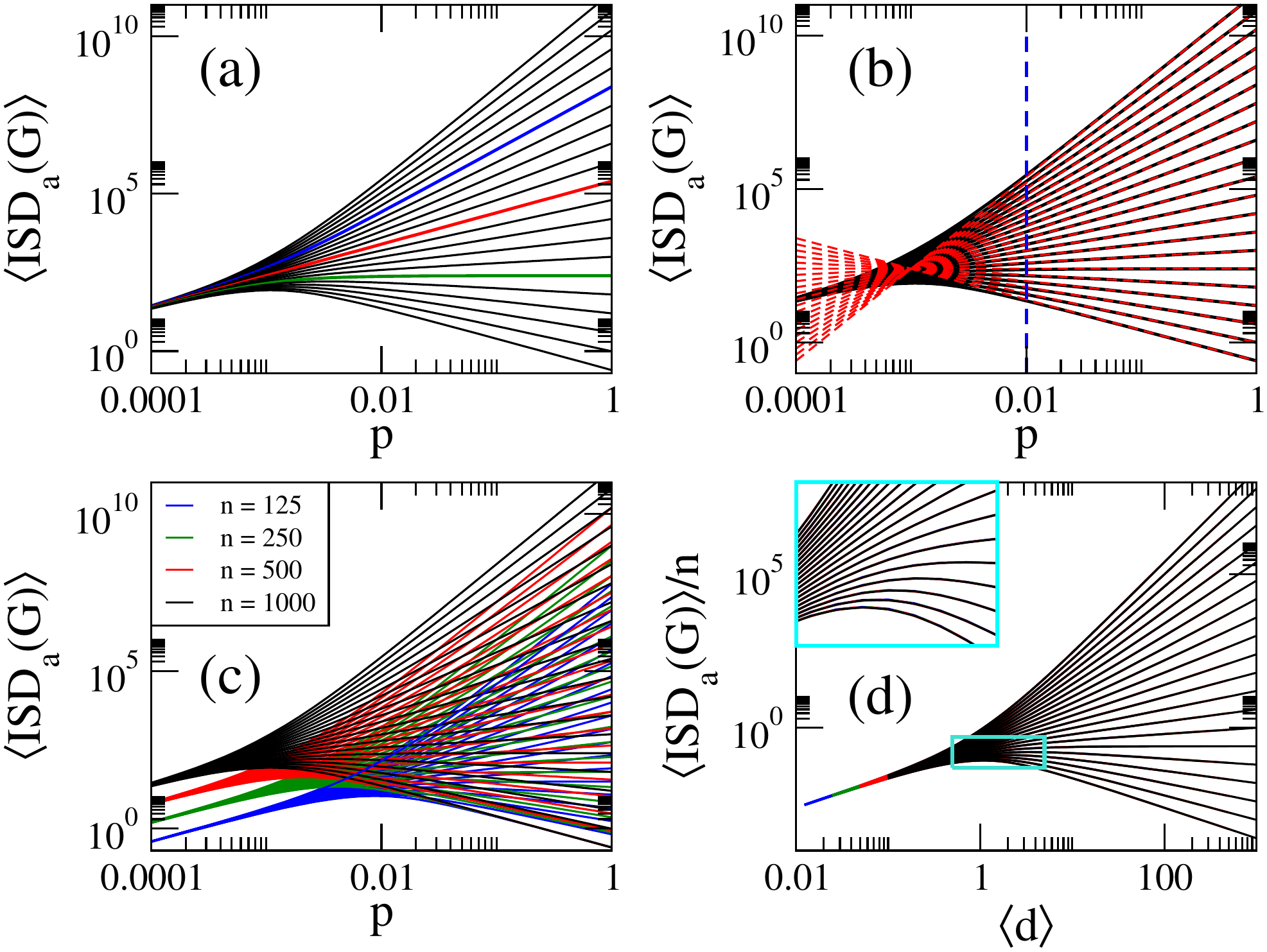}
\caption{\footnotesize{
(a,b) Average variable inverse sum deg index $\left< ISD_a(G) \right>$ as a 
function of the probability $p$ of Erd\H{o}s-R\'enyi graphs of size $n=1000$.
Here we show curves for $\alpha\in[-2,2]$ in steps of $0.2$ (from top to bottom).
Blue, red and green curves in (a) correspond to $a=-1$, $a=0$ and $a=1$, respectively.
The red dashed lines in (b) are Eq.~(\ref{avISDp}). The blue dashed line in (b)
marks $\left< d \right>= 10$.
(c) $\left< ISD_a(G) \right>$ as a function of the probability $p$ of 
ER graphs of four different sizes $n$.
(d) $\left< ISD_a(G) \right>/n$ as a function of the average degree $\left< d \right>$. 
Same curves as in panel (c). The inset in (d) is the enlargement of the cyan rectangle.
All averages are computed over $10^7/n$ random graphs.
}}
\label{Fig01}
\end{center}
\end{figure}

Also, from Fig.~\ref{Fig01}(a) we observe that the curves of $\left< ISD_a(G) \right>$ show three
different behaviors as a function of $p$ depending on the value of $a$:
For $a<a_0$, they grow for small $p$, approach a 
maximum value and then decrease when $p$ is further increased. 
For $a>a_0$, they are monotonically increasing functions of $p$.
For $a=a_0$ the curves saturate above a given value of $p$.
Here $a_0=1$.

Moreover, when $n p\gg 1$, we can write $d_u \approx d_v \approx \left<  d \right>$, with
\begin{equation}
\label{k}
\left< d \right> \approx (n-1)p .
\end{equation}
Therefore, for $n p\gg 1$, $\left< ISD_a(G) \right>$ is well approximated by:
\begin{equation}
\label{avISDp}
\left< ISD_a(G) \right> \approx \sum_{uv \in E(G)} \frac{1}{\left<  d \right>^a + \left<  d \right>^a} 
= |E(G)| \frac{1}{2\left<  d \right>^a} \approx  
\frac{n}{4} \left[ (n-1)p \right]^{1-a}.
\end{equation}
In Fig.~\ref{Fig01}(b), we show that Eq.~(\ref{avISDp}) (red-dashed lines) indeed describes well the data 
(thick black curves) for $np \ge 10$.

Now in Fig.~\ref{Fig01}(c) we show $\left< ISD_a(G) \right>$ 
as a function of the probability $p$ of ER random graphs of four different sizes $n$. 
It is quite clear from this figure that the blocks of curves, characterized by the different 
graph sizes, display similar curves but displaced on both axes. Thus, our next goal is to find the scaling 
parameters that make the blocks of curves to coincide.

First, we recall that the average degree $\left< d \right>$, see Eq.~(\ref{k}),
is known to scale both topological and spectral measures applied to ER graphs.
In particular, $\left< d \right>$ was shown to scale the normalized Randic index~\cite{MMRS20}, 
the normalized Harmonic~\cite{MMRS21} index, as well as several variable degree--based
indices~\cite{AIMS20} on ER graphs. Thus, we expect $\left< ISD_a(G) \right>\propto f(\left< d \right>)$.
Second, we observe in Fig.~\ref{Fig01}(c) that the effect of 
increasing the graph size is to displace the blocks of curves $\left< ISD_a(G) \right>$
vs.~$p$, characterized by the different graph sizes, 
upwards in the $y-$axis. Moreover, the fact that these blocks of curves, plotted in semi-log 
scale, are shifted the same amount on the $y-$axis when doubling $n$ is a clear signature 
of scalings of the form $\left< ISD_a(G) \right> \propto n^\beta$. 
By plotting $\left< ISD_a(G) \right>$ vs.~$n$ for given values of $p$ (not shown here) we conclude that 
$\beta=1$ for all $a$. 

Therefore, in Fig.~\ref{Fig01}(d) we plot 
$\left< ISD_a(G) \right>/n$ as a function of $\left< d \right>$ showing 
that all curves are now properly scaled; i.e.~the blocks of curves painted in different colors for different 
graph sizes fall on top of each other (see a detailed view in the inset of this figure). 
Moreover, following Eq.~(\ref{avISDp}), we obtain
\begin{equation}
\label{avISDk}
\frac{\left< ISD_a(G) \right>}{n} \approx \frac{1}{4} \left< d \right>^{1-a}.
\end{equation}
We have verified that Eq.~(\ref{avISDk}) is valid when $\left< d \right>\ge 10$.

\subsection{Inequalities of the average $ISD_a$ index on random graphs}

Most inequalities obtained in the previous Section are not restricted
to any particular type of graph. Thus, they should also be valid for random graphs and, moreover,
can be extended to index average values, as needed in computational studies of random graphs.

Now, in order to ease the computational validation of some of the inequalities derived in the previous 
Section, we:
\begin{itemize}
\item[(i)] write the right inequality of Eq.~(\ref{Eq1}) in Theorem~\ref{t:chi} as
\begin{equation}
0 \le \left< 2^{a-1} \chi_{_{-a}}(G) - ISD_a(G) \right> , \qquad \text{if } \, a > 1,
\label{Eq1av}
\end{equation}
\item[(ii)] write the left inequality of Eq.~(\ref{Eq2}) in Theorem~\ref{t:chi} as
\begin{equation}
0 \le \left< ISD_a(G) - 2^{a-1} \chi_{_{-a}}(G) \right> , \qquad \text{if } \, 0< a < 1,
\label{Eq2av}
\end{equation}
\item[(iii)] write the inequality of Eq.~(\ref{Eq3}) in Theorem~\ref{t:chi} as
\begin{equation}
0 \le \left< 2^{a-1} \chi_{_{-a}}(G) - ISD_a(G) \right> , \qquad \text{if } \, a < 0 ,
\label{Eq3av}
\end{equation}
\item[(iv)] write the inequality of Eq.~(\ref{Eq4}) in Theorem~\ref{t:m1} as
\begin{equation}
 \frac52 \, \left< m \right> \le \left< ISD_a(G) + M_1^{a+1}(G) \right>,
\qquad \text{if } \, a > 0,
\label{Eq4av}
\end{equation}
\item[(v)] write the inequality of Eq.~(\ref{Eq5}) in Theorem~\ref{t:m1} as
\begin{equation}
2 \, \left< m \right> \le \left< ISD_a(G) + M_1^{a+1}(G) \right>,
\qquad \text{if } \, a < 0,
\label{Eq5av}
\end{equation}
and
\item[(vi)] write the left inequality of Eq.~(\ref{Eq6}) in Theorem~\ref{t:M1bis} as
\begin{equation}
\left< m^2 \right>
\le \left< ISD_a(G) M_1^{a+1}(G) \right> .
\label{Eq6av}
\end{equation}
\end{itemize}

Therefore, in Figs.~\ref{Fig02}(a-f) we plot the r.h.s.~of the inequalities~(\ref{Eq1av}-\ref{Eq6av}), 
respectively, as a function of the probability $p$ of ER graphs of size $n=100$.

\begin{figure}[t!]
\begin{center}
\includegraphics[width=0.8\textwidth]{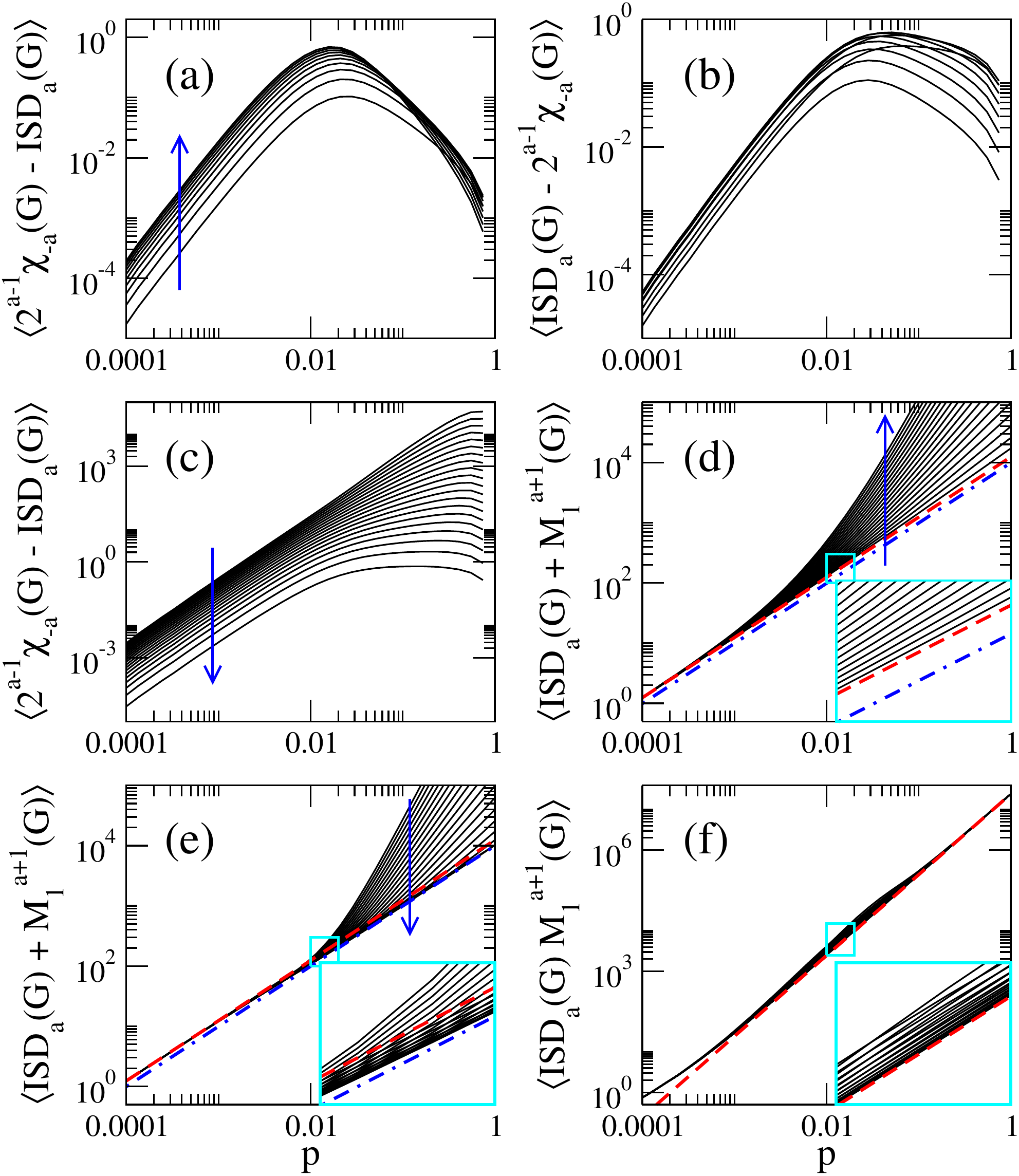}
\caption{\footnotesize{
In panels (a-f) we plot the r.h.s.~of relations~(\ref{Eq1av}-\ref{Eq6av}), respectively,
as a function of the probability $p$ of Erd\H{o}s-R\'enyi graphs of size $n=100$.
In (a) $a\in[1.1,2]$ in steps of 0.1, in (b) $a\in[0.1,0.9]$ in steps of 0.1, 
in (c) $a\in[-2,-0.1]$ in steps of 0.1, in (d) $a\in[0.1,2]$ in steps of 0.1,
in (e) $a\in[-2,-0.1]$ in steps of 0.1, and in (f) $a\in[-2,2]$ in steps of 0.2.
Dashed lines in (d,e) are $(5/2)\left< m \right>$ (red) and $2\left< m \right>$ (blue).
The red dashed line in (f) is $\left< m^2 \right>$. Here we used $\left< m \right>=n(n-1)p/2$.
Insets in (d-f) are enlargements of the cyan rectangles of the corresponding main panels.
Blue arrows in panels (a,c-e) indicate increasing $a$.
All averages are computed over $10^7/n$ random graphs.
}}
\label{Fig02}
\end{center}
\end{figure}

In particular, since the curves in Figs.~\ref{Fig02}(a-c) are all positive, the 
inequalities~(\ref{Eq1av}-\ref{Eq3av}) are easily validated. 
Now, in order validate inequalities~(\ref{Eq4av},\ref{Eq5av}) we include (as dashed lines) in 
both Fig.~\ref{Fig02}(d) and Fig.~\ref{Fig02}(e) the functions $(5/2)\left< m \right>$ vs.~$p$ (red) and 
$2\left< m \right>$ vs.~$p$ (blue); where we used $\left< m \right>=n(n-1)p/2$.
Then, we can clearly see that all curves $\left< ISD_a(G) + M_1^{a+1}(G) \right>$ vs.~$p$ in
Fig.~\ref{Fig02}(d) lie above the red dashed line, corresponding to $(5/2)\left< m \right>$;
while all curves $\left< ISD_a(G) + M_1^{a+1}(G) \right>$ vs.~$p$ in
Fig.~\ref{Fig02}(e) lie above the blue dot-dashed line, corresponding to $2\left< m \right>$.
This can be better appreciated in the enlargements shown in the panel insets.
Finally, in Fig.~\ref{Fig02}(f) we include, as a red dashed line, te function $\left< m^2 \right>$ vs.~$p$ to
clearly show that all curves $\left< ISD_a(G) M_1^{a+1}(G) \right>$ vs.~$p$ lie above it, as stated
in inequality~(\ref{Eq6av}). Moreover, note that the equality in (\ref{Eq6av}) is attained for $p\to 1$.
This is indeed expected since for $np\gg 1$ we can write
$$
\begin{aligned}
\left< ISD_a(G) M_1^{a+1}(G) \right> & =
\left< \sum_{uv \in E(G)} \frac{1}{d_u^a + d_v^a} \sum_{uv \in E(G)} d_u^a + d_v^a \right> 
\\ & \approx
\left< \sum_{uv \in E(G)} \frac{1}{\left<  d \right>^a + \left<  d \right>^a} \sum_{uv \in E(G)} \left<  d \right>^a + \left<  d \right>^a \right>
\\ & = \left< \frac{m}{\left<  d \right>^a + \left<  d \right>^a} m (\left<  d \right>^a + \left<  d \right>^a) \right>
= \left< m^2 \right> ,
\end{aligned}
$$
which we have observed to be valid for several graph sizes when $\left< d \right> \ge 10$ .

\section{Summary}

In this work we performed analytical and computational studies of the variable inverse sum deg index $ISD_a(G)$.
First, we analytically obtained new inequalities connecting $ISD_a(G)$ with other well--known topological indices
such as the Randi\'c index, the general sum-connectivity index, the geometric-arithmetic index, the 
arithmetic-geometric index, as well as the variable first Zagreb index.
Then, we computationally validated some of the obtained inequalities on ensembles of Erd\H{o}s-R\'enyi graphs 
$G(n,p)$ characterized by $n$ vertices connected independently with probability $p \in (0,1)$.
Additionally, we showed that the ratio $\bra ISD_a(G) \ket/n$ depends only on the average degree $\bra d \ket = (n-1)p$.

\begin{acknowledgements}
The research of W.C., J.M.R. and J.M.S. was supported by a grant from Agencia Estatal de Investigaci\'on (PID2019-106433GBI00/AEI/10.13039/501100011033), Spain.
J.M.R. was supported by the Madrid Government (Comunidad de Madrid-Spain) under the Multiannual Agreement with UC3M in the line of Excellence of University Professors (EPUC3M23), and in the context of the V PRICIT (Regional Programme of Research and Technological Innovation). 
\end{acknowledgements}

%
\section*{Conflict of interest}
The authors declare that they have no conflict of interest.



\end{document}